\documentclass{amsart}
\usepackage[margin=3cm]{geometry}
\usepackage{amssymb,amsmath,amsfonts,mathtools,latexsym}
\usepackage[pdftex]{hyperref} 
\usepackage{tikz-cd} 
\usepackage{enumerate} 
\newtheorem{theorem}{Theorem}[section]
\newtheorem{corollary}[theorem]{Corollary}

\newtheorem{lemma}[theorem]{Lemma}

\newtheorem{proposition}[theorem]{Proposition}

\begin{document}
\title[The Nilpotency of the Nil Metric $\mathbb{F}$-Algebras]{The Nilpotency of the Nil Metric $\mathbb{F}$-Algebras}
\keywords{normed fields, metric algebras, complete metric algebras, nil algebras, nilpotent algebras, graded algebras, PI-algebras, polynomial identities, Banach algebras, K\"othe's Problem.}
\subjclass[2020]{Primary 16R10; Secondary 46H99, 16W99, 16W50, 16R40, 12J05.}
\date{April 23, 2025}

\author[De Fran\c{c}a]{Antonio de Fran\c{c}a$^\dag$}
\address{Department of Mathematics, Federal University of Campina Grande, 58429-970 Campina Grande, Para\'iba, Brazil}
\email{\href{mailto: a.defranca@yandex.com}{a.defranca@yandex.com}}
\thanks{$^\dag$The author was partially supported by Para\'iba State Research Foundation (FAPESQ), Grant \#2023/2158.}

%
\begin{abstract}
Let $\mathbb{F}$ be a normed field. In this work, we prove that every nil complete metric $\mathbb{F}$-algebra is nilpotent when $\mathbb{F}$ has characteristic zero. This result generalizes Grabiner's Theorem for Banach algebras, first proved in 1969. Furthermore, we show that a metric $\mathbb{F}$-algebra $\mathfrak{A}$ and its completion $C(\mathfrak{A})$ satisfy the same polynomial identities, and consequently, if $\mathsf{char}(\mathbb{F})=0$ and $C(\mathfrak{A})$ is nil, then $\mathfrak{A}$ is nilpotent. Our results allow us to resolve K\"othe's Problem affirmatively for complete metric algebras over normed fields of characteristic zero.
\end{abstract}

\maketitle

%
\section{Introduction}

The motivation for this work arises from a theorem by S. Grabiner, published in \cite{Grab69}, in 1969, which asserts that ``\textit{if $B$ is a Banach algebra that is also a nil algebra, then $B$ is a nilpotent algebra}''. Recall that a \textit{Banach algebra} is an associative algebra over $\mathbb{R}$ (or $\mathbb{C}$) that is also a complete normed space with its norm satisfying $|xy|\leq|x|\cdot |y|$.   	%
Our goal is to study and answer the following question:

	\begin{center}
{\bf Problem$^\bigstar$:} 
\textit{Are nil metric $\mathbb{F}$-algebras nilpotent algebras, where $\mathbb{F}$ is a normed field?}
	\end{center}

We say that an associative algebra $\mathfrak{A}$ over a field $\mathbb{F}$ is \textit{a metric $\mathbb{F}$-algebra} if $\mathbb{F}$ is a field equipped with a norm which satisfies $|ab|\leq|a|\cdot|b|$, $\mathfrak{A}$ is a metric space, and its metric and operations (addition, product and multiplication by scalar) are compatible (in a suitable sense). 
%
%
%

Algebraic structures equipped with \textit{metrics} or \textit{norms} are interesting objects of study. 
For instance, I. Kaplansky showed in \cite{Kapl47} that ``\textit{a locally compact ring without divisors of zero is either connected or totally disconnected}''. 
In \cite{Kimu49}, N. Kimura exhibited an example of a normed ring with a left unit but no right unit, and proved that if $R$ is a Banach space and a ring with a unit in which multiplication is continuous, then $R$ is isomorphic to a normed ring where $|xy|\leq|x|\cdot|y|$. 
Already in \cite{Auro57}, S. Aurora proved that if $\mathfrak{R}$ is a metric ring such that $\Vert na \Vert=n\cdot \Vert a \Vert$ for any $n\in\mathbb{N}$ and $a\in \mathfrak{R}$, then there is an extension of $\mathfrak{R}$ which is a complete normed algebra over $\mathfrak{R}$ (see Corollary on page 1300). See also the works \cite{Auro58,Auro61} due to S. Aurora. 
In turn, in \cite{Dale81}, Dales showed that any commutative nil algebra with countable basis is normable, and that there exist a commutative nilpotent algebra and other commutative algebra with a countable basis which are not normable. 
%
Afterwards, V. M\"uller gave in \cite{Mull94} an overview of results concerning Banach nil, nilpotent and PI-algebras. See also \cite{Mull90} due to V. M\"uller. 
Posteriorly, H. Ochsenius and W.H. Schikhof exhibited in \cite{OchsSchi06} an example of a $\mathbb{K}$-normed space $E$ which is complete, but not a Baire space, where $\mathbb{K}$ is a valued field (see Remark on page 75). 
 For additional reading about normed (or metric) rings/algebras, see the books \cite{Naim72} and \cite{Rick60}. 

Another motivation for studying the above-mentioned \textbf{Problem}$^\bigstar$ is that Grabiner's Theorem provides an affirmative answer to Levitzky's Problem for Banach algebras. 
Levitzky's Problem is related to Kurosh's Problem. %
In \cite{Kuro41}, A.G. Kurosh, in analogy to Burnside Problem (in Group Theory), posed the following problem: ``\textit{is every algebraic algebra a locally finite algebra?}''. Recall that an algebra is locally finite if every finitely generated subalgebra is nilpotent. 
 Afterwards in \cite{Levi43}, J. Levitzky posed that ``\textit{is every nil ring a nilpotent ring?}''. This problem is known as \textit{Levitzky's Problem}. 
In turn, in \cite{Jaco45}, N. Jacobson presented results which reduce the study of Kurosh's Problem for algebraic algebras of bounded degree to the study of nil algebras of bounded degree. 
Already in \cite{Levi46}, J. Levitzky proved that ``\textit{each nil ring of bounded index is semi-nilpotent}''. Recall that a ring is called \textit{semi-nilpotent} if each finite set of elements in this ring generates a nilpotent ring. 
 Posteriorly, in \cite{Kapl48}, I. Kaplansky proved that any nil algebra satisfying a polynomial identity is locally finite. 
So, finally, the Dubnov-Ivanov-Nagata-Higman Theorem states that, being $\mathfrak{A}$ an associative $\mathbb{F}$-algebra such that $a^n=0$ for any $a\in\mathfrak{A}$, if $\mathsf{char}(\mathbb{F})$ is equal to $0$ or greater than $n$, then $\mathfrak{A}$ is nilpotent (see \cite{DubnIvan43}, \cite{Naga52} and \cite{Higm56}). 
%

In the context of rings with gradings, A. de Fran\c{c}a and I. Sviridova proposed in \cite{Mardua01} the following problem: ``\textit{given a monoid $\mathsf{S}$ and a ring $\mathfrak{R}$ with a finite $\mathsf{S}$-grading, does $\mathfrak{R}_e$ nil imply $\mathfrak{R}$ nil/nilpotent?}''. They proved that this problem finds a positive answer in the class of $\mathsf{f}$-commutative rings, and that every positive answer to this problem implies that K\"{o}the's Problem has a positive solution. 
This last problem was proposed by G. K\"{o}the in 1930 (see \cite{Koth30}), and since then, this conjecture has been confirmed in some classes of rings, but it still does not have a general solution. K\"{o}the's Problem asks ``\textit{if a ring $\mathfrak{R}$ has no nonzero nil ideals, then does $\mathfrak{R}$ have nonzero one-sided nil ideals?}''. The K\"{o}the's Problem and Kurosh Problem are related. 
For further reading, as well as an overview, about K\"{o}the's Problem, see \cite{Smok01}, \cite{Mardua01} and their references. 

%
%

This work has only one section of results, namely \S2, which is divided into 3 subsections.
 In the initial part of section \S2, we present the main tools used: \textit{normed field}, \textit{metric algebras} and its (main) properties, and \textit{Baire Category Theorem}. And we also present some results. The main one is 

\vspace{0.15cm}
\noindent{\bf Theorem \ref{5.1}}: Let $\mathbb{F}$ be a normed field of characteristic zero, and $\mathfrak{A}$ a complete metric $\mathbb{F}$-algebra. If $\mathfrak{A}$ is nil, then $\mathfrak{A}$ is nilpotent.
\vspace{0.15cm}

In subsections \S2.1 and \S2.3, being $\mathbb{F}$ a normed field and $\mathfrak{A}$ a metric $\mathbb{F}$-algebra, we show that \textit{$\mathfrak{A}$ and its completion have the same polynomial identities} ({\bf Proposition \ref{5.6}}), and \textit{if $\mathsf{char}(\mathbb{F})=0$ and the completion of $\mathfrak{A}$ is nil, then $\mathfrak{A}$ is nilpotent} ({\bf Corollary \ref{5.5}}).

We finish this work deducing that the K\"othe's Problem has a positive answer in the class of complete metric algebras over normed fields of characteristic zero ({\bf Corollary \ref{5.9}}).

%
\section{The Nilpotency of Nil Complete Metric Algebras}

%
%

Let us begin this text by introducing some pertinent definitions and elementary results concerning normed fields and metric algebras.

Let $\mathbb{F}$ be any field. We say that $\mathbb{F}$ is a \textit{normed field} if there exists a map $|\ . \ |:\mathbb{F}\rightarrow \mathbb{R}_+$  which has the usual properties of a norm and satisfies $|-\lambda|=|\lambda|$ and $|\lambda\gamma|\leq |\lambda|\cdot|\gamma|$ for all $\lambda, \gamma \in \mathbb{F}$. 
When $\mathsf{char}(\mathbb{F})=0$, let us assume that $|n \cdot \lambda|= n\cdot|\lambda|$ for all $n\in\mathbb{N}$ and $\lambda\in\mathbb{F}$.

Now, let $(\mathfrak{A},\mathsf{d})$ be a metric space which is also an associative $\mathbb{F}$-algebra, where $\mathbb{F}$ is a normed field. We say that \textit{$\mathfrak{A}$ is a metric $\mathbb{F}$-algebra} (with respect to the metric $\mathsf{d}$) when there is $\gamma\in\mathbb{R}_+^*$ such that
	\begin{enumerate}[i)]
	\item $\mathsf{d}(a+c,b+c)\leq \gamma \mathsf{d}(a,b)$ for any $a,b,c\in\mathfrak{A}$;
	\item $\mathsf{d}(\lambda a,\lambda b)\leq \gamma |\lambda|\mathsf{d}(a,b)$ for any  $\lambda\in\mathbb{F}$ and $a,b\in\mathfrak{A}$;
	\item $\mathsf{d}(ac,bc)\leq \gamma \mathsf{d}(c,0)\mathsf{d}(a,b)$ and $\mathsf{d}(ca,cb)\leq \gamma \mathsf{d}(c,0)\mathsf{d}(a,b)$ for any $a,b,c\in\mathfrak{A}$.
	\end{enumerate}
Observe that $\gamma\geq1$ necessarily. When $(\mathfrak{A},\mathsf{d})$ is a complete metric space, $\mathfrak{A}$ is called \textit{a complete metric $\mathbb{F}$-algebra}. The notions of \textit{sequences}, \textit{convergent sequence}, \textit{limit of a sequence}, and \textit{the completion of a metric algebra} are introduced in the standard way and their usual properties can be easily verified (see \cite{Auro57} and \cite{Kapl77}). In particular, 
%
%
it is not difficult to see that the items i), ii) and iii) imply that the operations $+, \times:\mathfrak{A}\times\mathfrak{A}\rightarrow \mathfrak{A}$ and $\cdot:\mathbb{F}\times\mathfrak{A}\rightarrow \mathfrak{A}$ are continuous. Consequently, given sequences $\{a_n\}_{n\in\mathbb{N}},\{b_n\}_{n\in\mathbb{N}}\in\mathfrak{A}$ and $\{\lambda_n\}_{n\in\mathbb{N}}\in\mathbb{F}$ converging to $a,b\in\mathfrak{A}$ and $\lambda\in\mathbb{F}$, respectively, we have that 
$\lim_{n\in\mathbb{N}} (a_n+b_n)= (a+b)$, 
$\lim_{n\in\mathbb{N}} a_n b_n = ab$ and 
$\lim_{n\in\mathbb{N}} \lambda_n a_n= \lambda a$.

The main result of this work is the following theorem. 
 It generalizes Grabiner's Theorem on the nilpotency of Banach nil algebras, which was established in \cite{Grab69}.

\begin{theorem}\label{5.1}
Let $\mathbb{F}$ be a normed field of characteristic zero, and $\mathfrak{A}$ a complete metric $\mathbb{F}$-algebra. If $\mathfrak{A}$ is nil, then $\mathfrak{A}$ is nilpotent.
\end{theorem}

Before proving this theorem, let us present and demonstrate two facts.

\begin{lemma}\label{mainlemma}
Let $\mathbb{F}$ be any field, $\mathfrak{A}$ an $\mathbb{F}$-algebra, and $\mathfrak{A}[t]$ the polynomial algebra in the variable $t$ with coefficients in $\mathfrak{A}$. Given any polynomial $p(t)\in\mathfrak{A}[t]$ of degree $n$, if there exist $\lambda_1,\dots,\lambda_{n+1}\in\mathbb{F}$ pairwise distinct such that $p(\lambda_i)=0$ for all $i=1,\dots,n+1$, then $p(t)\equiv0$.
\end{lemma}
\begin{proof}
Take any $p(t)=a_0+a_1 t+ a_2 t^2+\cdots+ a_n t^n \in\mathfrak{A}[t]$. Given any $\lambda_1,\dots,\lambda_{n+1}\in\mathbb{F}$, it follows that
\begin{equation}\nonumber
	\begin{bmatrix} p(\lambda_1)\\
		p(\lambda_2)\\
		\vdots\\
		p(\lambda_{n+1})\\
	\end{bmatrix}
=
	\Delta
	\cdot
	\begin{bmatrix} 
	a_0 \\
	a_1 \\
	\vdots\\
	a_n\\
\end{bmatrix} \ , \
\mbox{ where } \
\Delta=\begin{bmatrix} 
		1 & \lambda_1 & \lambda_1^2 & \cdots & \lambda_1^n\\
		1 & \lambda_2 & \lambda_2^2 & \cdots & \lambda_2^n\\
	\vdots & \vdots & \vdots & \ddots &\vdots \\
		1 & \lambda_{n+1} & \lambda_{n+1}^2 & \cdots & \lambda_{n+1}^n\\
	\end{bmatrix} \ .
\end{equation}
When $\lambda_i\neq\lambda_j$ for all $i\neq j$, $\Delta$ is a Vandermonde matrix, which is invertible (see \S4.3, p.124, in \cite{HoffKunz71}). Consequently, assuming that $\lambda_r$'s are pairwise distinct such that $p(\lambda_r)=0$, it follows that $a_r=0$ for all $r=0,1,\dots,n$, and so the result follows.
\end{proof}

%
%

Now, recall that Baire Category Theorem states that, given a complete metric space $Y$, any countable collection $\{Y_n\}$ of closed sets of $Y$ each of which has empty interior in $Y$, their union $\bigcup Y_n$ also has empty interior in $Y$ (see \cite{Munk00}, Theorem 48.2, p.296). Consequently, if $Y$ is a nonempty complete metric space which is equal to the countable union $\bigcup Y_n$, we have that at least one of the sets $\overline{Y}_n$ has a nonempty interior (see Exercise 1, p.298, \cite{Munk00}), where $\overline{Y}_n$ is the closure of $Y_n$.

\begin{lemma}\label{maintheorem}
Let $\mathbb{F}$ be a normed field with $\mathsf{char}(\mathbb{F})=0$, and $\mathfrak{A}$ a complete metric $\mathbb{F}$-algebra. If $\mathfrak{A}$ is nil, then $\mathfrak{A}$ is nil of bounded index.
\end{lemma}
\begin{proof}

For each $i\in\mathbb{N}$, consider the subset of $\mathfrak{A}$ given by $B_i=\{a \in\mathfrak{A}: a^i=0\}$. Note that $\mathfrak{A}=\bigcup_{i\in\mathbb{N}}B_i$ and $B_i\subseteq B_{i+1}$ for all $i\in\mathbb{N}$. Let us show that $B_n=\mathfrak{A}$ for some $n\in\mathbb{N}$.

\textbf{Claim 1:} Each $B_i$ is a closed set of $\mathfrak{A}$. Indeed, given $a\in\overline{B_i}$, there is a sequence $\{a_r \}_{r\in\mathbb{N}}\in B_i$ converging to $a$, and hence, as the product $\times:\mathfrak{A}\times\mathfrak{A}\rightarrow \mathfrak{A}$ is continuous, it follows that $a^i = \left( \lim_{r\in\mathbb{N}} a_r \right)^i=\lim_{r\in\mathbb{N}} (a_r^i)=0$, and so $a\in B_i$, and $B_i$ is closed. 
Now, by Baire Category Theorem, there exist $a_0\in\mathfrak{A}$, $r\in\mathbb{R}_+^*$ and $n\in\mathbb{N}$ such that $B_r(a_0)\subseteq \overline{B}_n=B_n$, where $B_r(a_0)=\{a\in\mathfrak{A}: \mathsf{d}(a_0,a)<r \}$. 
\textbf{Claim 2:} For any $b\in\mathfrak{A}$, the polynomial $p(t)=(a_0+t(b-a_0))^n$ of $\mathfrak{A}[t]$ is null. In fact, let $\gamma\in\mathbb{R}_+^*$ such that $\mathsf{d}(a+a_2,a_1+a_2)\leq\gamma\mathsf{d}(a,a_1)$ and $\mathsf{d}(\lambda a, \lambda a_1)\leq\gamma|\lambda|\mathsf{d}(a,a_1)$ for any $a,a_1,a_2\in\mathfrak{A}$ and $\lambda\in\mathbb{F}$. Since $|k\cdot\lambda|= k\cdot|\lambda|$ for all $k\in\mathbb{N}$ and $\lambda\in\mathbb{F}$ because $\mathsf{char}(\mathbb{F})=0$, we have that $|(k \cdot 1_{\mathbb{F}})^{-1}| = k^{-1} \cdot |1_{\mathbb{F}}| \neq0$ for all $k\in\mathbb{N}$. 
Take $k_0\in\mathbb{N}$ such that $k_0>|1_{\mathbb{F}}|\gamma^2 r^{-1}\mathsf{d}(b-a_0,0)$. So, for all $k\geq k_0$, it follows that 
\begin{equation}\nonumber
\mathsf{d}(a_0,a_0+(k \cdot 1_{\mathbb{F}})^{-1}(b-a_0))\leq\gamma\mathsf{d}(0,(k \cdot 1_{\mathbb{F}})^{-1}(b-a_0)) \leq\gamma^2 k^{-1}|1_{\mathbb{F}}|\mathsf{d}(0,b-a_0)< r \ .
\end{equation}
 From this, $a_0+(k \cdot 1_{\mathbb{F}})^{-1}(b-a_0)\in B_r(a_0)$ for all positive integer $k\geq k_0$, and hence, $p((k \cdot 1_{\mathbb{F}})^{-1})=(a_0+(k \cdot 1_{\mathbb{F}})^{-1}(b-a_0))^n=0$ for all $k\geq k_0$. 
Thus, the Claim 2 follows from Lemma \ref{mainlemma}, and consequently, $b^n= p(1_{\mathbb{F}})= 0$. Therefore, $\mathfrak{A}\subseteq B_n$, and so $\mathfrak{A}$ is nil of bounded index.
\end{proof}

%

According to what was shown above, we can now establish a proof of Theorem \ref{5.1}.

\begin{proof}[\textbf{Proof of Theorem \ref{5.1}}]
By Lemma \ref{maintheorem}, we have that $\mathfrak{A}$ is nil of bounded index. From this, there is $n\in\mathbb{N}$ such that $a^n=0$ for any $a\in\mathfrak{A}$, and hence, since $\mathfrak{A}$ is an $\mathbb{F}$-algebra with $\mathsf{char}(\mathbb{F})=0$, it follows from Dubnov-Ivanov-Nagata-Higman Theorem (see \cite{DubnIvan43}, \cite{Naga52} and \cite{Higm56}) that $\mathfrak{A}$ is nilpotent.
\end{proof}
%

%
\subsection{The Non-Complete Case}

Let $\mathfrak{A}$ be a metric $\mathbb{F}$-algebra, where $\mathbb{F}$ is a normed field. It is well-known that any metric space has a completion (see \cite{Kapl77}, Theorem 57, p.92). Let $(C(\mathfrak{A}), \overline{d})$ be the completion of $\mathfrak{A}$, that is, $(C(\mathfrak{A}), \bar{d})$ is a complete metric space which contains $(\mathfrak{A},d)$ as a subset dense. Hence, given any $a,b \in C(\mathfrak{A})$, there exist sequences $\{a_n\}_{n\in\mathbb{N}}$ and $\{b_n\}_{n\in\mathbb{N}}$ in $\mathfrak{A}$ converging to $a$ and $b$, and consequently, $\bar{d}(a,b)=\lim_{n\in\mathbb{N}}d(a_n,b_n)$. From this, by continuity of operations ``$+$'', ``$\times$'' and ``$\cdot$'', it is not difficult to show that $C(\mathfrak{A})$ is a complete metric $\mathbb{F}$-algebra. 

In \cite{CiolFreiGonc17}, L. Cioletti, J.A. Freitas and D.J. Gon{\c{c}}alves (2017) proved that, given a normed PI-algebra $A$ over $\mathbb{R}$ (or $\mathbb{C}$), if $C(A)$ is nil, then $A$ is nilpotent. The next result generalizes this last one.
%
%

\begin{corollary}\label{5.5}
Let $\mathbb{F}$ be a normed field with $\mathsf{char}(\mathbb{F})=0$, and $\mathfrak{A}$ a metric $\mathbb{F}$-algebra. If $C(\mathfrak{A})$ is nil, then $\mathfrak{A}$ is nilpotent.
\end{corollary}
\begin{proof}
By Theorem \ref{5.1}, we have that $C(\mathfrak{A})$ is nilpotent, and so there is $n\in\mathbb{N}$ such that $(C(\mathfrak{A}))^n=\{0\}$. Since that $\mathfrak{A}\subseteq C(\mathfrak{A})$, it follows that $\mathfrak{A}^n=\{0\}$, and consequently, $\mathfrak{A}$ is a nilpotent algebra.
\end{proof}

%
\subsection{The T-ideals of Identities of a Metric Algebra and its Completion}
Let $\mathbb{F}$ be a field and $\mathbb{F}\langle X \rangle$ the free associative algebra, generated freely by the set $X=\{x_1,x_2,\dots\}$, a countable set of indeterminants. Let $\mathfrak{A}$ be an associative $\mathbb{F}$-algebra. We say that $g=g(x_1,\dots,x_d)\in\mathbb{F}\langle X \rangle$ is a \textit{polynomial identity} of $\mathfrak{A}$, denoted by $g\equiv0$ in $\mathfrak{A}$, if $g(a_1,\dots,a_d)=0$ for any $a_1,\dots,a_d\in\mathfrak{A}$. We denote by $\mathsf{T}(\mathfrak{A})$ the set $\{g\in\mathbb{F}\langle X \rangle : g\equiv0  \mbox{ in }  \mathfrak{A}\}$ of all polynomial identities of $\mathfrak{A}$. For further reading, as well as an overview, on free algebras and algebras with polynomial identities, we indicate the books \cite{Dren00} and \cite{GiamZaic05}. 
The below result improves Proposition 1 from \cite{CiolFreiGonc17}, which states that ``\textit{if $A$ is a normed PI-algebra over $\mathbb{R}$ (or $\mathbb{C}$), then $\mathsf{T}(C(A))=\mathsf{T}(A)$}''.

\begin{proposition}\label{5.6}
Let $\mathbb{F}$ be a normed field. Any metric $\mathbb{F}$-algebra and its completion satisfy the same polynomial identities.
\end{proposition}
\begin{proof}
Let $\mathfrak{A}$ be a metric $\mathbb{F}$-algebra. Obviously $\mathsf{T}(C(\mathfrak{A}))\subseteq \mathsf{T}(\mathfrak{A})$, because $\mathfrak{A} \subseteq C(\mathfrak{A})$. 	
Reciprocally, take any $f(x_1,\ldots,x_d) \in \mathsf{T}(\mathfrak{A})$ and $a_1,\dots,a_d \in C(\mathfrak{A})$. Since $C(\mathfrak{A})$ is the completion of $\mathfrak{A}$, there exist convergent sequences $\{a_{1n}\}_{n\in\mathbb{N}},\dots,\{a_{dn}\}_{n\in\mathbb{N}}\in\mathfrak{A}$ such that 
$\lim_{n\in\mathbb{N}} a_{in}=a_i$ 
for all $i=1,\dots, d$. From this, since the three operations ``$+$'', ``$\times$'' and ``$\cdot$'' of $\mathfrak{A}$ are continuous, we have that 
\begin{equation}\nonumber
f(a_1,\dots,a_d)=f(\lim_{n\in\mathbb{N}} a_{1n},\dots,\lim_{n\in\mathbb{N}} a_{dn})=\lim_{n\in\mathbb{N}} f( a_{1n},\dots, a_{dn})=\lim_{n\in\mathbb{N}} 0=0 \ .
\end{equation}
Therefore, $f \in \mathsf{T}(C(\mathfrak{A}))$, and the result follows.
\end{proof}

A relevant question is whether \textit{two metric $\mathbb{F}$-algebras satisfying the same polynomial identities are necessarily isomorphic}. The answer is {\bf no}! In fact, let $\mathfrak{A}=\mathbb{Q}$ and $\mathfrak{B}=\mathbb{R}$ be two metric $\mathbb{Q}$-algebras, both equipped with their usual metrics. As $C(\mathbb{Q})=\mathbb{R}$, by the previous proposition, it follows that $\mathfrak{A}$ and $\mathfrak{B}$ satisfy the same polynomial identities. Obviously, $\mathfrak{A}$ and $\mathfrak{B}$ are not isomorphic. Also, $\mathbb{R}$ and $\mathbb{C}$ are non-isomorphic complete metric $\mathbb{R}$-algebras with the same polynomial identities. Therefore, two (complete) metric $\mathbb{F}$-algebras that have the same polynomial identities are not necessarily isomorphic.

%
\subsection{Graded Metric Algebras}
Let $\mathsf{S}$ be a monoid and $\mathfrak{A}$ an $\mathbb{F}$-algebra. An \textit{$\mathsf{S}$-grading on $\mathfrak{A}$} is a decomposition $\Gamma: \mathfrak{A} = \bigoplus_{\xi\in \mathsf{S}} \mathfrak{A}_\xi$ that satisfies $\mathfrak{A}_\xi \mathfrak{A}_\zeta \subseteq \mathfrak{A}_{\xi\zeta}$, for any $\xi, \zeta \in \mathsf{S}$, where $\mathfrak{A}_\xi$'s are vector $\mathbb{F}$-subspaces of $\mathfrak{A}$ (see the books \cite{GiamZaic05} and \cite{NastOyst11}). 
The study of \textit{algebras with gradings} plays an important role in obtaining properties of these algebras, in the non-graded sense. In \cite{BergCohe86}, J. Bergen and M. Cohen (1986) showed that if $\mathsf{G}$ is a finite group with neutral element $e$ and $\mathfrak{A}$ is a $\mathsf{G}$-graded algebra, then $\mathfrak{A}_e$ is a PI-algebra iff $\mathfrak{A}$ is a PI-algebra. See also  \cite{BahtGiamRile98}, due to Yu. Bahturin, A. Giambruno and D. Riley (1998).
An interesting example of a property obtained from a grading is provided by Theorem 3.9 in \cite{Mardua01}, due to A. de Fran\c{c}a and I. Sviridova (2022). It states that ``\textit{if $\mathsf{S}$ is a left cancellative monoid and $\mathfrak{A}$ is an (associative) $\mathbb{F}$-algebra with a finite $\mathsf{S}$-grading, then $\mathfrak{A}_e$ is nilpotent iff $\mathfrak{A}$ is nilpotent}''. 
Let us use this last result to prove the theorem below, and still to give a positive answer (for metric $\mathbb{F}$-algebras) to the question asked by A. de Fran\c{c}a and I. Sviridova in \cite{Mardua01}, namely, they posed the following problem: ``\textit{given a ring $\mathfrak{R}$ with a finite $\mathsf{S}$-grading, does $\mathfrak{R}_e$ nil imply $\mathfrak{R}$ nil/nilpotent?}''. This problem implies K\"othe's Problem (\cite{Mardua01},  Theorem 4.15).

%
%
\begin{theorem}
Let $\mathsf{S}$ be a left cancellative monoid, $\mathbb{F}$ a normed field of characteristic zero and $\mathfrak{A}$ an $\mathbb{F}$-algebra with a finite $\mathsf{S}$-grading. Suppose that $\mathfrak{A}_e$ is a metric $\mathbb{F}$-algebra. If $C(\mathfrak{A}_e)$ is nil, then $\mathfrak{A}$ is a nilpotent algebra. In particular, if $\mathfrak{A}_e$ is a nil complete metric $\mathbb{F}$-algebra, then $\mathfrak{A}$ is a nilpotent algebra.
\end{theorem}
\begin{proof}
By Corollary \ref{5.5} (resp. Theorem \ref{5.1}), when $C(\mathfrak{A}_e)$ is nil (resp. $\mathfrak{A}_e$ is a nil complete metric $\mathbb{F}$-algebra), it follows that $\mathfrak{A}_e$ is a nilpotent algebra. 
By Theorem 3.9 from \cite{Mardua01}, the result follows.
\end{proof}

\begin{corollary}\label{5.8}
Let $\mathfrak{A}$ be a metric $\mathbb{F}$-algebra with a finite $\mathsf{S}$-grading, where $\mathsf{S}$ is a cancellative monoid and $\mathbb{F}$ is a normed field of characteristic zero. If $C(\mathfrak{A}_e)$ is nil, then $\mathfrak{A}$ is nilpotent.
\end{corollary}



\begin{corollary}[K\"othe's Problem for Metric Algebras]\label{5.9}
The K\"othe's Problem has a positive answer for any complete metric algebra over a normed field of characteristic zero.
\end{corollary}
%
%
%
%


\bibliographystyle{amsplain}


\end{document}